\documentclass[10 pt]{amsart}
\usepackage{amscd,amsmath,amsthm,amssymb}
\usepackage{enumerate,varioref}
\usepackage{epsfig}
\usepackage{graphicx}
\newtheorem{thm}{Theorem}

\newtheorem{lem}[thm]{Lemma}

\theoremstyle{definition}

\theoremstyle{remark}
\newtheorem{ex}[thm]{Example}
\newtheorem{rem}[thm]{Remark}
\newtheorem{Qu}[thm]{Question}
\numberwithin{equation}{subsection}

\newcommand{\Z}{\mathbb{Z}}
\newcommand{\C}{\mathbb{C}}

\begin{document}

\title{The Haar State on $SU_q(N)$}
\author{Clark Alexander}

\email{gcalex@temple.edu} \maketitle \tableofcontents

\section{Introduction}
\indent When attacking the problem of generalizing index theorems
from low dimensional objects to higher dimensional analogous
objects, one encounters many different generalization methods.  Sadly,
many of these generalizations lead one down a fruitless path.  The
following sections are the original work of the author in all cases
where $N>2.$  It is the author's hope that these ''correct"
generalizations shed some light on how to compute with higher
dimensional compact matrix quantum groups.  The second section shows
that the decomposition from classical representation theory yields a
useful result in the quantum case.  The third section uses in a
specific way the decomposition of $SU_q(2)$ into direct summands
indexed by pairs of integers to compute the Haar state in the
case $N=2.$  Section four shows that the techniques exploited in
section three do not generalize to the higher dimensional cases
without some arduous re-indexing.  However, particular techniques
from section three will come together in a rather fascinating way to
give a Haar state on $SU_q(N)$ that is strikingly similar to
that on $SU_q(2)$ when the proper re-indexing has occured.  A corollary to the result obtained here is that $SU_q(N)$ is given an orthonormal basis as a vector space in the form
\begin{equation}
|\ell m n\rangle = \frac{1}{\sqrt{h(t_{ij}^{\ell*}t_{ij}^{\ell})}}t_{ij}^{\ell}.
\end{equation}

This work is intended in part as part two of [A].

\section{Peter-Weyl Type Decomposition of $SL_q(N)$}
\indent  The Peter-Weyl type decomposition of $SL_q(N)$ is nearly
identical to the classical case of $SL(N,\C).$  For $N\in
\mathbb{N}$ consider the algebras $$\mathcal{O}(K_N):=
\C[z_1,\dots,z_N]/\langle z_1\cdots z_N =1\rangle.$$  These are the
function algebras of the maximal tori of $SL(N,\C).$  In the case
$N=2$ we have exactly the Laurent polynomials in one variable.  One
may, however, put the structure of a Hopf algebra on
$\mathcal{O}(K_N)$ by setting $\Delta(z_i)=z_j\otimes z_j,$
$\epsilon(z_j)=1$, $S(z_j)= z_j^{-1}.$\\
\begin{rem}
Since $\prod z_j =1 $ it is only necessary to use $N-1$ such $z_j$
as $$z_N=z_1^{-1}\cdots z_{N-1}^{-1}.$$
\end{rem}
\indent Now consider the homomorphisms
$$\phi : SL_q(N) \rightarrow \mathcal{O}(K_N) $$
given by
\begin{equation}
\phi(u_{i,j}) = \delta_{ij}z_j
\end{equation}

These homomorphisms are in fact Hopf algebra homomorphisms.  Now
consider the homomorphisms
\begin{eqnarray*}L_K: SL_q(N) \rightarrow
\mathcal{O}(K_N)\otimes SL_q(N),\\
R_K: SL_q(N) \rightarrow SL_q(N)\otimes \mathcal{O}(K_N).
\end{eqnarray*}

Given by
\begin{eqnarray*}
L_K = (\phi \otimes id)\circ \Delta\\
R_K = (id\otimes \phi)\circ \Delta,
\end{eqnarray*}

and define the sets
\begin{equation}\mathcal{A}\lbrack \alpha,\beta
\rbrack := \{ x\in SL_q(N) | L_K(x) = z^{\alpha}\otimes x , R_K(x) =
x\otimes z^{\beta}\}.
\end{equation}
Here, $\alpha$ and $\beta$ are multi-indices and $z^{\alpha} =
\prod_{i=1}^{N-1} z_i^{\alpha_i}.$

These sets $\mathcal{A}[\alpha,\beta]$ shall be known as the
$\alpha$-left, $\beta$-right invariant sets of
$\mathcal{O}(SL_q(N)).$  Elements of $\mathcal{A}[0,0]$ shall be
known as $\emph{K bi-invariant}$ elements.\\
\indent The present goal is to show that these sets form a
decomposition of $SL_q(N)$ and that the only elements which garner
nontrivial Haar measure are those belonging to
$\mathcal{A}[0,0].$\\
\indent Presently, only $SL_q(2)$ shall receive attention.  Once
this decomposition is established for $N=2$ the proper
generalizations are easy to make.  In the case of $SL_q(2)$ one sees
that the sets $\mathcal{A}[m,n]$ are indexed by pairs of integers.
Indeed, the homomorphism $\phi$ acts by
\begin{eqnarray*}
u_{11}\mapsto z,& u_{12}\mapsto 0,\\
u_{21}\mapsto 0, & u_{22}\mapsto z^{-1}.
\end{eqnarray*}

In order to check that these sets yield a decomposition one needs to
check two things:
\begin{enumerate}
\item All the generators fall into a single set.\\
\item Multiplication of elements falls into a set i.e. if
$x\in\mathcal{A}[m,n]$ and $y\in\mathcal{A}[r,s]$ then
$xy\in\mathcal{A}[p,q]$ for some $p,q.$
\end{enumerate}

\begin{rem}
For the rest of the exposition of $SL_q(2)$ we write, as is common
\begin{eqnarray*}
u_{11}=a,&u_{12}=b,\\
u_{21}=c, & u_{22}=d.
\end{eqnarray*}
\end{rem}

\begin{lem}
In the case of $SL_q(2)$ one has
\begin{enumerate}
\item[(a)] the generators $a,b,c,d$ belong to distinct sets, and\\
\item[(b)] $\mathcal{A}[m,n] \cdot \mathcal{A}[r,s] \subset
\mathcal{A}[m+r,n+s].$
\end{enumerate}
\end{lem}

\begin{proof}
Part $(a)$ is shown by direct computation.  Only $a$ will be shown
here, the rest are done in precisely the same manner.
\begin{eqnarray*}
L_K(a) = &(\phi\otimes id)\Delta(a)\\ = &(\phi\otimes id)(a\otimes a
+ b\otimes c)\\ = & \phi(a)\otimes a + \phi(b)\otimes c \\
= & z\otimes a\\
&\\
R_K(a) = & a\otimes \phi(a) + b\otimes\phi(c)\\
= & a\otimes z
\end{eqnarray*}
Hence $a\in\mathcal{A}[1,1].$ Likewise
$b\in\mathcal{A}[-1,1],c\in\mathcal{A}[1,-1],d\in\mathcal{A}[-1,-1].$\\
As for $(b)$ one needs to utilize the fact that $\phi, L_K,$ and
$R_K$ are homomorphisms.\\
Let $x\in\mathcal{A}[m,n], y\in\mathcal{A}[r,s]$ then
\begin{eqnarray}
L_K(xy) = L_K(x)L_K(y)  = (z^m\otimes x)(z^r\otimes y) =
z^{m+r}\otimes xy\nonumber\\
R_K(xy) = R_K(x)R_K(y)  = (x\otimes z^n)(y\otimes z^s) = xy\otimes
z^{n+s}
\end{eqnarray}
\end{proof}

Therefore, one may now write
\begin{equation}
\mathcal{O}(SL_q(2)) = \bigoplus_{m,n\in\Z}\mathcal{A}[m,n].
\end{equation}

When one attempts to replicate the proof for higher dimensions,
there are few, if any, stopping blocks.  In fact, the generators
$u_{i,j}$ for $SL_q(N)$ are prescribed to $\mathcal{A}[\alpha,\beta]$ in the same way.  Using the
coproduct when $N>2$ is marginally more tedious, but the
homomorphisms kill off more elements than before.  Furthermore, when
checking the second condition, the only thing left to worry about is
how to deal with multi-indices.  This, however, gives no trouble in
the actual computation.  Therefore, one may also write
\begin{equation}
\mathcal{O}(SL_q(N)) = \bigoplus_{\alpha,\beta\in \Z^{N-1}}\mathcal{A}[\alpha,\beta].
\end{equation}

What has happened is that the map $\phi$ sends $SL_q(N)$ into the
coordinate algebra of the maximal torus of $SL(N,\C)$ in direct
analogy with the classical Peter-Weyl decomposition theorem.

\begin{rem}
Depending on the presentation of information shown to the reader,
the generalization from $N=2$ to $N>2$ should be easy.  However,
there is one beautiful anomaly that occurs in the case $N=2.$ Namely
one can show for every $\ell$ that
\begin{equation}t_{i,j}^{\ell} \in
\mathcal{A}[-2i,-2j]
\end{equation}
where the $t_{i,j}^{\ell}$ are the matrix corepresentations from
before (cf. [A]).  This is only possible because the indices of the
decomposition are integers and not elements in an integer lattice.
This particular piece of information is propitious when computing
the Haar state on $SU_q(2).$
\end{rem}

\section{The Haar State on $SU_q(2)$}
\indent Woronowicz graced the mathematical world with a proof that
there exists a unique bi-invariant linear functional satisfying
\begin{equation}
h(x)\cdot I = (id\otimes h)\Delta(x) = (h\otimes id)\Delta(x); h(1)=1.
\end{equation}

In the case of $SL_q(2)$ one can easily determine
$h(t_{i,j}^{\ell}) = 0$ when $\ell>0.$  One might wonder if there
are any nontrivial nonvanishing elements under $h.$  Indeed, there
are, but one needs to be clever to find them.

\begin{lem}
The only nonvanishing elements under $h$ are the K bi-invariant
elements.
\end{lem}
\begin{proof}
Let $x\in\mathcal{A}[m,n].$  Then using the bi-invariance of $h$ and
$L_K,R_K$ one obtains $z^mh(x)=h(x)=h(x)z^n.$  More explicitly one
has
\begin{eqnarray*}
z^mh(x) = &(id\otimes h)(z^m\otimes x)\\ = &(id\otimes
h)(\phi\otimes id)\Delta(x)
\end{eqnarray*}
But $(id\otimes h)$ and $(\phi\otimes id)$ commute so that
\begin{eqnarray*}
z^mh(x) = & (id\otimes h)(\phi\otimes id)\Delta(x) \\
= & (\phi\otimes id)(id\otimes h)\Delta(x)\\
=& \phi(1)h(x) = h(x).
\end{eqnarray*}
One treats $h(x)z^n$ similarly.  Therefore $h(x)=0$ if
$(m,n)\neq(0,0).$
\end{proof}

Equipped with this information, this first obvious choices to find a nontrivial
measure are $ad$ and $bc.$  Moreover, in the case of $SU_q(2)$ one
has an algebra equipped with a $*$-product and finds that
\begin{equation}x\in\mathcal{A}[m,n] \iff x^*\in\mathcal{A}[-m,-n].\end{equation}
This information becomes more prevalent in the higher dimensional
cases. Another important piece of information to keep at bay is
\begin{equation}x\in\mathcal{A}[m,n] \iff
S(x)\in\mathcal{A}[-n,-m]\end{equation} from which one may easily
derive the relationship between $x$ and $x^*$.\\
\indent On $SU_q(2)$ the $*$-product yields $b^*=-qc.$  Therefore,
the first element examined here will be $-qbc=:\zeta.$  Utilizing
Woronowicz's equations, one finds
\begin{eqnarray}
h(\zeta) = &(id\otimes h)\circ\Delta(\zeta)\nonumber\\
=& (id\otimes h)(-q)(a\otimes b + b\otimes d)(c\otimes a+ d\otimes
c)\nonumber\\
=& -q(id\otimes h)(ac\otimes ba + ad\otimes bc + bc\otimes da +
bd\otimes dc)\nonumber\\
=&ad h(\zeta) + \zeta h(da)\nonumber\\
=& (1-\zeta)h(\zeta) + \zeta h(1-q^{-2}\zeta)\nonumber\\
&\implies h(\zeta) = \frac{1-q^{-2}}{1-q^{-4}}.
\end{eqnarray}

One important point to realize before going through further
computations is that many elements vanish under $h.$  It behooves
one to project from $\mathcal{O}(SL_q(2))$ to $\mathcal{A}[0,0]$
before beginning any computations.    Klimyk and Schmudgen have
provided a few horrendous formulae for the general reader in this
vein.  Letting $P$ be the aforemention projection; here they are:
\begin{eqnarray*}
(id\otimes
P)\circ\Delta(\zeta^n)=\sum_{i+j=n}\left[\begin{array}{c}n\\i
\end{array}\right]^2_{q^{-2}}q^{2ij}\zeta^j(\zeta;q^2)_i\otimes\zeta^i(q^{-2}\zeta;q^{-2})_j\\
h(\zeta^n)\cdot I=\sum_{i+j=n}\left[\begin{array}{c}n\\i
\end{array}\right]^2_{q^{-2}}q^{2ij}\zeta^j(\zeta;q^2)_ih(\zeta^i(q^{-2}\zeta;q^{-2})_j)\\
h(\zeta^n) = \frac{1-q^{-2n}}{1-q^{-2(n+1)}}h(\zeta^{n-1}).
\end{eqnarray*}

By noting in $SL_q(2)$ that $\mathcal{A}[0,0] = \C[\zeta]$
one now knows how to compute the Haar state within $SL_q(2).$
Moreover one knows how to compute $h(x^*x)$ and $h(xx^*)$ for any
$x\in SU_q(2).$  The present goal then shall be to compute a more
general Haar state on $SU_q(2)$ using matrix
corepresentations.\\

\indent Consider the two Hermitian forms on $\mathcal{O}(SU_q(2))$
given by
\begin{eqnarray}
\langle x,y\rangle_L = h(x^*y), & \langle x,y\rangle_R = h(xy^*), &
x,y\in\mathcal{O}(SU_q(2)).
\end{eqnarray}

Since one should desire scalar products to be sesquilinear, a choice
is necessary to determine which of these hermitian forms is an inner
product.  As presented here scalar products shall be linear in the
first variable, making $\langle \cdot,\cdot\rangle_L$ and
$\overline{\langle \cdot,\cdot\rangle}_R$ scalar products on the
vector space $\mathcal{O}(SU_q(2)).$
\begin{rem}
Certain special properties of $h$ and $\langle\cdot,\cdot\rangle$
necessitate comment.  The Haar state while linear, is not
central.  That is to say in general
\begin{equation}
h(xy)\neq h(yx).
\end{equation}
Therefore one should like to have a method of interpolating between
the two.  The preferred method is to look for an automorphism
$\vartheta$ such that
\begin{equation}
h(xy)=h(\vartheta(y)x).
\end{equation}
It is here that one gets a glimpse of why $N=2$ is so special.  It
this case one can solve $\vartheta$ directly on the generators and
find that
\begin{eqnarray*}
\vartheta(a) = q^2a, & \vartheta(b) = b\\
\vartheta(c) = c, & \vartheta(d) =q^{-2}d.\\
\end{eqnarray*}
What is remarkable is that
\begin{eqnarray}
\vartheta(x) = q^{m+n}x; & \forall x\in\mathcal{A}[m,n].
\end{eqnarray}
And in particular at $N=2$
\begin{equation}
\vartheta(t_{i,j}^{\ell}) = q^{-2(i+j)}t_{i,j}^{\ell}.
\end{equation}
 No such nicety is available for $N>2$ as
the indices $m,n$ are points in an integer lattice rather than
integers themselves. This
problem will be resolved later.\\
\indent Two further remarks from definitions;
\begin{enumerate}
\item[(a)] $\langle xz,y\rangle_R = \langle x,yz^*\rangle_R$ and
similarly $\langle zx,y\rangle_L = \langle x,z^*y\rangle_L$\\
\item[(b)] $\langle x,y\rangle_L = \langle \vartheta(y),x\rangle_R.$
\end{enumerate}
\end{rem}

\begin{thm}
\begin{enumerate}
\item [(i)]The decomposition of $\mathcal{O}(SU_q(2))$ into matrix
corepresentations is an orthogonal decomposition under
$\langle\cdot,\cdot\rangle_L$ and $\langle\cdot,\cdot\rangle_R$\\
\item [(ii)]The matrix corpresentations yield the following formulae
for $h.$
\begin{eqnarray}
\langle t_{i,j}^{\ell},t_{i,j}^{\ell}\rangle_L =
\frac{q^{-2i}}{[2\ell+1]_q}\\
\langle t_{i,j}^{\ell},t_{i,j}^{\ell}\rangle_R =
\frac{q^{2j}}{[2\ell+1]_q}
\end{eqnarray}
\end{enumerate}
\end{thm}
\begin{proof} (cf. [KS])
For part $(i)$ it has already been established that $\langle
t_{i,j}^{\ell},t_{r,s}^k\rangle = 0$ if $(i,j)\neq(r,s).$  What is
left to establish is orthogonality when $\ell\neq k.$  This argument
reduces to Schur's lemma for Hopf algebras.\\
\indent Consider a $(2\ell+1)\times(2k+1)$ matrix $M.$  Define
$\tilde M := h(T^{\ell}MT^{k*})$ and $\tilde M' := h(T^{\ell
*}MT^{k}).$ Then $\tilde M=0$ and $\tilde M'=0$ when $\ell\neq k.$
This assertion is shown by again considering the invariance
properties of $h.$
\begin{eqnarray*}
T^{\ell}\tilde M T^{k*} &=& (id\otimes h)((T^{\ell}\otimes
I)(I\otimes T^{\ell}) M (I\otimes T^{k*})(T^{k*}\otimes I))\\
&=&((id\otimes h)\circ\Delta)(T^{\ell}MT^{k*})\\
&=&h(T^{\ell}MT^{k*})=\tilde M
\end{eqnarray*}
Thus one obtains $$T^{\ell}\tilde M = \tilde M T^{k}.$$  That is to
say that $\tilde M$ intertwines irreducible corepresentations.  By
Schur's lemma, the only invariant subspaces are empty or the whole
space.  Hence when $\ell\neq k$ $\tilde M = 0.$  The same argument
shows this for $\tilde M'.$ Schur's lemma gives even more
information however.  Not only is the invariant subspace for $\tilde
M$ the whole space, but $\tilde M$ and $\tilde M'$ take the special
forms 
\begin{eqnarray*} 
\tilde M = \alpha I, & \tilde M' = \alpha' I
& \alpha,\alpha'\in\C.
\end{eqnarray*}\\

The quantities one now seeks are $\langle
t_{i,j}^{\ell},t_{i,j}^{\ell}\rangle_L = \alpha'_i$ and $\langle
t_{i,j}^{\ell},t_{i,j}^{\ell}\rangle_R = \alpha_j.$  But one already has a
relation between these two numbers in the guise of $$\langle
t_{i,j}^{\ell},t_{i,j}^{\ell}\rangle_L =
\langle\vartheta(t_{i,j}^{\ell}),t_{i,j}^{\ell}\rangle_R.$$ Hence
\begin{equation}
\alpha'_i = q^{-2(i+j)}\alpha_j.
\end{equation}
Moreover there exists $\alpha$ so that $\alpha =
q^{2i}\alpha_i'=q^{-2j}\alpha_j$ for all $i,j.$  However, from the
computation above
\begin{eqnarray}
h(\zeta^{2\ell}) &=& h((b^*b)^{2\ell})\nonumber\\
&=& \langle t_{i,-\ell}^{\ell},t_{i,-\ell}^{\ell}\rangle
\nonumber\\
&=& \alpha_{-\ell} = \frac{q^{-4\ell}(1-q^{-2})}{1-q^{-4\ell-2}}
\end{eqnarray}
Therefore, one obtains
$$\alpha = \frac{q^{-2\ell}(1-q^{-2})}{1-q^{-4\ell-2}}$$
and
\begin{eqnarray}
\alpha_j = \frac{q^{2j}}{[2\ell+1]_q},\nonumber\\
\alpha'_i = \frac{q^{-2i}}{[2\ell+1]_q}
\end{eqnarray}
\end{proof}

\section{Generalizing to $SU_q(N)$}
One of the many conveniences ascribed to the case $N=2$ is the fact
that the automorphism $\vartheta$ may be written
$\vartheta(x)=q^{m+n}x$ when $x\in\mathcal{A}[m,n].$  Perhaps one of
the first steps in generalizing to the $N>2$ case should be to
produce a similar automorphism that accounts for the noncommutative
property of $h.$  One should like to have
\begin{eqnarray}
h(xy)=h(\vartheta(y)x) & \forall x,y\in\mathcal{O}(SU_q(N)).
\end{eqnarray}

The first issue encountered here is that $\mathcal{O}(SU_q(N))$
cannot be reduced to $N$ generators as in the case $N=2.$  In fact,
since the $*$-structure in $SU_q(N)$ involves quantum determinants
of cofactors $\mathcal{O}(SU_q(N))$ properly has $N^2$ generators.
With this in mind, the desired automorphism $\vartheta$ requires
$n^2$ parameters to be fully determined.  There are only a handful
of properties that one can guarantee of $\vartheta$, namely
\begin{enumerate}
\item If $\vartheta(x) = \beta x$ then $\vartheta(x^*) = \beta^{-1}x^*.$
This insures that the determinant relations hold on $\mathcal{O}(SU_q(N))$\\
\item When $x^*x=xx^*$ then $\vartheta(x)=x.$  Specifically this
happens at $x= t_{-\ell,\ell}^{\ell}$ and $x=t_{\ell,-\ell}^{\ell}.$
Note that when $N\neq 2$ then $\ell$ does not increment by $1/2$,
but rather by $\left(\begin{array}{c}N+k-1\\N-1\end{array}\right)$
halves at the $k$th step.
\end{enumerate}

The form $\vartheta(t_{i,j}^{\ell}) = q^{-2(i+j)}t_{i,j}^{\ell}$
from $\mathcal{O}(SU_q(2))$ fortunately yields an acceptable
automorphism in the higher cases.  What one needs to check in this
case is that this particular automorphism coincides with commutation
relations on $SU_q(N).$\\
\begin{ex}
Consider the following necessities of $h$ and their correlations
with relations on $SU_q(N).$

\begin{eqnarray}
\sum_{j=1}^N u_{1,j}u_{1,j}^* =1, && \sum_{j=1}^N
q^{-2(j-1)}u_{1,j}^*u_{1,j} = 1,\\
\sum_{i=1}^N u_{i,1}^*u_{i,1}, && \sum_{i=1}^N
q^{2(i-1)}u_{i,1}u_{i,1}^*,\nonumber\\
h(\sum_{j=1}^N u_{1,j}u_{1,j}^*) &=& \sum_{j=1}^N
h(u_{1,j}u_{1,j}^*)
=1,\nonumber\\
h(\sum_{j=1}^N q^{-2(j-1)}u_{1,j}^*u_{1,j})& =& \sum_{j=1}^N
q^{-2(j-1)}h(u_{1,j}^*u_{1,j})=1.\nonumber
\end{eqnarray}
This seems to suggest that $h$ varies directly with the sub-indices
of the generators.  Fortunately this is the case when $N=2$. Another
important clue derived from these equations is that when using the
left or right invariance of $h$ the coproducts will yield unsightly
equations involving scalars hitting elements of the algebra which
have specific relations.  For example when trying to compute
$h(u_{1,N}u_{1,N}^*)$ one arrives at
\begin{equation}
h(u_{1,N}u_{1,N}^*)\cdot I = \sum_{j=1}^N u_{1,j}u_{1,j}^*
h(u_{j,N}u_{j,N}^*)
\end{equation}
Clearly it is the case that $h(u_{1,N}u_{1,N}^*)\neq 0$ so one must
account for the fact that $\sum_{j}u_{1,j}u_{1,j}^*=1.$  What one 	67.03
must conclude is that $h(u_{1,N}u_{1,N}^*) = h(u_{j,N}u_{j,N}^*)$
for every $j\in\{1,\dots,N\}.$
\end{ex}

In a similar way, one can play all the tricks in computing relations
between $h(u_{i,j}u_{i,j}^*)$ and $h(u_{i,j}^*u_{i,j}).$  The
relations can be listed as follows:\\

\begin{enumerate}
\item $h(u_{i,j}u_{i,j}^*) = \langle u_{i,j},u_{i,j}\rangle_R$  is constant in $j$\\
\item $h(u_{i,j}^*u_{i,j}) = \langle u_{i,j},u_{i,j}\rangle_L$ is constant in $i$\\
\end{enumerate}

It is now convenient to move into computations with matrix
corepresentations.  Here one
should like to have the automorphism $\vartheta$ in hand.  Then one
needs to check $\vartheta(t_{i,j}^{\ell})= q^{k}t^{\ell}_{i,j}$
against the given relations on $h(u_{i,j}u_{i,j}^*).$  One will see
after a short computation
\begin{equation}
\vartheta(t_{i,j}^{\ell}) = q^{-2(i+j)}t_{i,j}^{\ell}.
\end{equation}

This is exactly the form of $\vartheta$ from $N=2.$ Then using the
invariance of $h$ one finds
\begin{equation}
h(t_{i,\ell}^{\ell} t_{i,\ell}^{\ell*}) = \sum_{k=-\ell}^{\ell}
h(t_{i,k}^{\ell}t_{i,k}^{\ell*})t_{k,\ell}^{\ell}t_{k,\ell}^{\ell*}
\end{equation}

\begin{ex} Looking at a quick example for $SU_q(4)$ one has
\[
u_{11}^*u_{11} + q^{-2}u_{12}^*u_{12} +q^{-4}u_{13}^*u_{13} +q^{-6}u_{14}^*u_{14} = 1. 
\]

Applying $h$ to both sides one and recognizing $h(u_{1j}^*u_{1j})=h(u_{11}^*u_{11})$ for all $j$ one arrives at
\begin{eqnarray}
&h(u_{1j}^*u_{1j})(1+q^{-2}+q^{-4}+q^{-6}) =& 1\\
&h(u_{1j}^*u_{1j}) =& \frac{1}{q^{-3}(q^3+q+q^{-1}+q^{-3})}\nonumber\\
&h(u_{1j}^*u_{1j}) =& \frac{q^3}{[4]_q}.\nonumber
\end{eqnarray}

Noting in  $SU_q(4)$ that $u_{1j}= t_{-3/2,j-5/2}^{3/2}$ one arrives at 
\begin{equation}
\langle t_{-3/2,\tilde j}^{3/2},t_{-3/2,\tilde j}^{3/2} \rangle_L = \frac{q^{2(3/2)}}{[2(3/2)+1]_q} = \frac{q^{-2i}}{[2\ell+1]_q}.
\end{equation}

\end{ex}

Putting all the steps together one needs to use $\vartheta$, the
respective constancy conditions in $i$ and $j$, the quantum
determinant relations on $\mathcal{O}(SU_q(N))$, and the bivariance
of $h$ to arrive at
\begin{eqnarray}
\langle t_{i,j}^{\ell},t_{i,j}^{\ell}\rangle_R =
\frac{q^{2j}}{\sum_{k=-\ell}^{\ell}q^{2k}} =
\frac{q^{2j}}{[2\ell+1]_q}\\
\langle t_{i,j}^{\ell},t_{i,j}^{\ell}\rangle_L =
\frac{q^{-2i}}{\sum_{k=-\ell}^{\ell}q^{2k}} =
\frac{q^{-2i}}{[2\ell+1]_q}.\\
\end{eqnarray}

These are the desired formulae for $h$ in $\mathcal{O}(SU_q(N))$ for
any $N.$  The difference in the higher dimensional cases is simply
the indexing on $\ell.$  Rather elementary combinatorics come into
play to aid one in the discovery that successive representations of
$SU_q(N)$ need not exist for each half integer.

\section{Concluding Remarks}

While the Haar State has been studied by several authors, the succinctness of the presentation at hand is new.  The advantage in reorganizing matrix corepresentations to depend on a single parameter yields a result that looks identical to the case which is explicitly computable when $N=2$.  One interesting consequence of the combinatorial re-indexing is a conjecture concerning possible spin states in higher dimensions.

\begin{Qu}
Are possible spin states of (theoretical) particles in dimensions with $SU(N)$ symmetries restricted to taking values in
\[
\{\frac{1}{2}(\left(\begin{array}{c}N+k-1\\k\end{array}\right)-1) \}?
\]
\end{Qu}

This is a question the author hopes to explore soon.

The work which remains is to extend the methods developed here to generalize the Dirac operator from $SU_q(2)$ to $SU_q(N)$ hopefully in the style of [DLSvSV].  The first step has been achieved and one may write an orthnormal basis for $SU_q(N)$ in the form 
\[
|\ell mn\rangle = \frac{1}{\sqrt{h(t^{\ell*}_{ij}t^{\ell}_{ij})}}t^{\ell_{ij}} = q^{i}[2\ell+1]_q^{1/2}t^{\ell}_{ij}
\]
so that
\[
\langle \ell'm'n'|\ell mn\rangle = \delta_{\ell'\ell}\delta_{m'm}\delta_{n'n}.
\]

\end{document}